\documentclass[11pt,british,a4paper]{amsart}

\IfFileExists{test.sty}{\usepackage[print]{tms}}{
\usepackage[dvipsnames]{xcolor}
\definecolor{url}{RGB}{0,113,185} %%% CUSTOM-BLUE %%%
\definecolor{link}{RGB}{160,0,0} %%% CUSTOM-RED %%%
\PassOptionsToPackage{ocgcolorlinks=true,citecolor=link,linkcolor=link,urlcolor=url,backref=page}{hyperref}
\PassOptionsToPackage{paperwidth=460pt,paperheight=700pt,height=590pt,width=360pt}{geometry}
}
\usepackage{mlmodern,dsfont}
\let\mathbb\mathds
\usepackage{microtype,amssymb,mathtools,hyperref}
\usepackage[T1]{fontenc}
\usepackage[marginratio=1:1]{geometry}
\usepackage[capitalize]{cleveref}
\usepackage[dvipsnames]{xcolor}
\usepackage{graphicx}

\numberwithin{equation}{section}

\newtheorem{thm}{Theorem}[section]
\newtheorem{cor}[thm]{Corollary}
\newtheorem{ppn}[thm]{Proposition}
\newtheorem{lem}[thm]{Lemma}
\theoremstyle{remark}
\newtheorem{rmk}[thm]{Remark}
\newtheorem{ex}[thm]{Example}
\theoremstyle{definition}
\newtheorem{defn}[thm]{Definition}

\renewcommand{\Re}{\mathrm{Re}}
\newcommand{\e}{\mathrm{e}}
\renewcommand{\i}{\mathrm{i}}
\newcommand{\C}{\mathbb{C}}
\newcommand{\R}{\mathbb{R}}
\renewcommand{\d}{\mathrm{d}}
\newcommand{\con}{{}^{[g]}\nabla}
\newcommand{\gnabla}{{}^g\nabla}
\newcommand{\bx}{\mathbf{X}}
\newcommand{\E}{\mathbb{E}}
\newcommand{\tr}{\operatorname{tr}}
\newcommand{\seq}{\left(P_{\ell}\right)_{\ell \geqslant 2}}

\title[The degree of a classical minimal surface]{The Schwarzian derivative and the degree of a classical minimal surface}

\author[T.~Mettler]{Thomas Mettler}
\author[L.~Poerschke]{Lukas Poerschke}
\address{Faculty of Mathematics and Computer Science, UniDistance Suisse, Brig, Switzerland}
\email{thomas.mettler@fernuni.ch, mettler@math.ch}
\address{Faculty of Mathematics and Computer Science, UniDistance Suisse, Brig, Switzerland}
\email{lukpoerschke@gmail.com}

\date{22nd July 2024}

\begin{document}

\begin{abstract}
Using the Schwarzian derivative we construct a sequence $\seq$ of mero\-mor\-phic differentials on every non-flat oriented minimal surface in Euclidean $3$-space. The  differentials $\seq$ are invariant under all deformations of the surface arising via the Weierstrass representation and depend on the induced metric and its derivatives only. A minimal surface is said to have degree $n$ if its $n$-th differential is a polynomial expression in the differentials of lower degree. We observe that several well-known minimal surfaces have small degree, including Enneper's surface, the helicoid/catenoid and the Scherk -- as well as the Schwarz family. Furthermore, it is shown that locally and away from umbilic points every minimal surface can be approximated by a sequence of minimal surfaces of increasing degree.     
\end{abstract}

\maketitle

\section{Introduction}

In~\cite{MR3667429}, a meromorphic quadratic differential $P_2$ is introduced on every non-flat oriented minimal surface $M$ in Euclidean $3$-space $\E^3$. The differential $P_2$ arises as a conservation law for a certain curvature entropy functional and is hence called the~\textit{entropy differential}. In this note we show that $P_2$ is the first element in a geometrically natural sequence $\seq$ of meromorphic differentials on $M$, where $P_{\ell}$ has degree $\ell$, that is, $P_{\ell}$ is a section of the $\ell$-th tensorial power of the canonical bundle of $M$. The entropy differentials $\seq$ arise as certain higher order Schwarzian derivatives of the stereographically projected Gauss map $G : M \to \mathbb{C}\cup \left\{\infty\right\}$; where we compute the Schwarzian derivative and its higher order descendants with respect to the Levi-Civita connection of the flat metric $\sqrt{-K}g$. %This flat metric is defined on the complement $M^{\prime}=M\setminus \mathcal{U}$ of the umbilic locus $\mathcal{U}$. 
Here $g$ denotes the induced metric on $M$ and $K\leqslant 0$ its Gauss curvature. The differentials $\seq$ satisfy the following properties:
 
\begin{thm}\label{mainthma}
Let $\bx : M\to \E^3$ be a non-flat minimal immersion of the oriented surface $M$. Denoting by $\seq$ the induced differentials, we have:
\begin{itemize}
\item[(i)] the differential $P_{\ell}$ is holomorphic away from umbilic points and extends mero\-mor\-phi\-cal\-ly to all of $M$ with a pole of order $\ell$ at the umbilic points;
\item[(ii)] the differential $P_{\ell}$ depends on the induced metric $g$ and its derivatives up to order $(\ell+4)$ only. In particular, the differentials $\seq$ are the same for all members of the associated family of minimal immersions of $\bx : M \to \E^3$;  
\item[(iii)] the differentials $\seq$ are invariant under Goursat transforms of the minimal immersion $\bx : M \to \E^3$; 
\item[(iv)] the differentials $\seq$ are invariant under (constant) rescaling of $\bx : M \to \E^3$. 
\end{itemize}
\end{thm}
Having the infinite sequence $\seq$ at hand, we may say that a minimal immersion $\bx : M \to \E^3$ has \textit{degree $n$} if $P_{n}$ vanishes identically or if $P_n$ is a polynomial expression in the lower order differentials, that is,
$$
P_{n}=\psi_n(P_2,\ldots,P_{n-1}),
$$
where $\psi_n$ is a weighted-homogeneous polynomial of degree $n$ which we call the~\textit{algebraic type} of the immersion. We observe that several classical minimal surfaces are surfaces of low degree. Using \cite{MR3667429}, it follows that -- up to Euclidean motion and the deformations listed in \cref{mainthma} -- open subsets of Enneper's surface are the only surfaces of degree $2$ and open subsets of the helicoid and catenoid are the only surfaces of degree $3$. A complete description of degree four surfaces is not feasible. However, we prove that if a degree four minimal surface admits an umbilic point, then it satisfies an equation of the form $P_4+a_n(P_2)^2=0$, where
$$
a_n=\frac{12(n+2)^2}{(3n^2+4n)}, \quad n \in \mathbb{N}. 
$$
The equation $P_4+a_n(P_2)^2=0$ is satisfied by the Enneper surfaces with higher dihedral symmetry group $D_{n+1}$. Furthermore, the sequence $a_n$ converges to $4$ as $n$ goes to infinity and we show that at this limit coefficient there exists a $1$-parameter family of immersed singly periodic degree four minimal deformations of the plane which keep the lines $\left\{z=0,y=\frac{k\pi}{2}\,|\, k \in \mathbb{Z}\right\}$ fixed and so that the induced metric admits a Killing vector field that is the real part of a holomorphic vector field. Likewise, a degree five minimal surface admitting an umbilic point satisfies $P_5+2a_nP_3P_2=0$. We show that the smallest admissible coefficient $216/7$ is realized by the Schwarz family (including the P- and D-surface and the Gyroid). Furthermore, the limit coefficient $8$ is realized by the Scherk family. Also, at degree seven we encounter the $k$-noids of Jorge \& Meeks~\cite{MR683761}. 

%Furthermore, we show that the entropy differentials up to order $n-1$ together with the Gauss curvature and its exterior derivative constitute the complete germ of a minimal surface metric of degree $n$. In particular, minimal immersions of any prescribed algebraic type $\psi_n$ exist and the local moduli space of such immersions is identified with the half-space of dimension $2n-3$. 

We also prove that locally and away from umbilic points a minimal surface can be approximated by a sequence of minimal surfaces of increasing degree. 

\begin{thm}\label{thmb}
Let $\bx : M \to \mathbb{E}^3$ be a non-flat minimal immersion of an oriented surface $M$. Then for every point $p$ in the complement $M^{\prime}$ of the umbilic locus there exists a neighbourhood $U_p$ and a sequence of minimal immersions $\bx_n : U_p \to \mathbb{E}^3$ with $\bx_n$ having degree $n$ so that $\lim_{n\to \infty}\Vert\bx_n-\bx\Vert_{g_{Eucl}}=0$ locally uniformly.  
\end{thm}

The Schwarzian derivative arises most naturally in the context of projective differential geometry. In \cref{appa} we discuss the relation between the definition of the Schwarzian derivative used in this note and the usual definition of projective differential geometry. 

\subsection*{Acknowledgements} This article is based on the doctoral thesis of L.P.. The authors were partially supported by the priority programme SPP 2026 ``Geometry at Infinity'' of DFG. The authors are grateful to Jacob Bernstein for several helpful discussions regarding the content of this article. The authors also express their gratitude to the anonymous referees for their  reports, which helped to enhance the clarity of this paper.

\section{Preliminaries}

\subsection{Weyl connections}

Let $(M,[g])$ be an oriented surface equipped with a conformal structure $[g]$. Recall that a Weyl connection for $[g]$ (or $[g]$-conformal connection) is a torsion-free connection $\nabla$ on $TM$ preserving $[g]$, that is, its parallel transport maps are angle preserving with respect to the conformal structure $[g]$. A torsion-free connection $\nabla$ on $TM$ preserves $[g]$ if and only if for some -- and hence any --  Riemannian metric $g \in [g]$, there exists a $1$-form $\beta \in \Omega^1(M)$, so that 
\begin{equation}\label{confconid}
\nabla g=2\beta\otimes g. 
\end{equation} 
Conversely, it follows from Koszul's identity that for every pair $(g,\beta)$ there exists a unique affine torsion-free connection satisfying~\eqref{confconid}, which is 
$$
{}^{(g,\beta)}\nabla=\gnabla+g\otimes \beta^{\#}-\beta\otimes \mathrm{Id}-\mathrm{Id}\otimes \beta, 
$$
where $\gnabla$ denotes the Levi-Civita connection of the metric $g$ and $\beta^{\#}$ denotes the $g$-dual vector field to $\beta$. Note that for every smooth function $u$ on $M$, the pair $(\e^{2u}g,\beta+\d u)$ determines the same Weyl connection as the pair $(g,\beta)$
$$
{}^{(\exp(2u)g,\beta+\d u)}\nabla={}^{(g,\beta)}\nabla, 
$$
as can easily be verified using the identity (c.f.~\cite[Theorem 1.159]{MR867684})
$$
{}^{\e^{2u}g}\nabla={}^g\nabla-g\otimes{}^g\nabla u+\d u \otimes \mathrm{Id}+\mathrm{Id}\otimes \d u.
$$
We will use the notation $\con$ to denote a Weyl connection for $[g]$ and simply write $\nabla$ when the conformal structure is clear from the context. 

Let $J$ denote the complex structure on $M$ induced by $[g]$ and the orientation. Clearly, a torsion-free connection preserves $[g]$ if and only if it preserves $J$ and we may therefore also think of the Weyl connections for $[g]$ as torsion-free $J$-linear connections, that is, torsion-free connections on $TM$ whose parallel transports maps are $J$-linear. Consequently, a Weyl connection $\nabla$ induces connections on all tensorial powers of the canonical bundle $L=T^{1,0}M^*$ of $(M,J)$. By standard abuse of notation, we will denote these connections by $\nabla$ as well, so that for all $\ell \in \mathbb{Z}$ we have first order differential operators
$$
\nabla : \Gamma(L^{\ell})\to \Omega^1(M,L^{\ell}),
$$
sending smooth sections of $L^{\ell}$ to $L^{\ell}$-valued $1$-forms on $M$. Since $L^{\ell}$ is a complex vector bundle, the $L^{\ell}$-valued $1$-forms on $M$ decompose
$$
\Omega^1(M,L^{\ell})=\Omega^{1,0}(M,L^{\ell})\oplus \Omega^{0,1}(M,L^{\ell})
$$
into $(1,\! 0)$ and $(0,\! 1)$ parts. Note that we may canonically identify $\Omega^{1,0}(M,L^{\ell})$ with $\Gamma(L^{\ell+1})$ so that the $(1,\! 0)$ part of $\nabla$ may be thought of as a differential operator
$$
\nabla^{1,0} : \Gamma(L^{\ell}) \to \Gamma(L^{\ell+1}). 
$$ 
For a smooth section $\sigma \in \Gamma(L^{\ell})$, we write $\sigma^{\prime}=\nabla^{1,0}\sigma$ as well as $\sigma^{\prime\prime}=(\sigma^{\prime})^{\prime}$ and likewise for higher order derivatives. 

For what follows it will be convenient to have local coordinate expressions for the differential operators $\nabla^{1,0}$. To this end let $z : U \to \C$ be a local holomorphic coordinate on $M$. Extending $\nabla$ complex-linearly to the complexified tangent bundle $TM\otimes \C$, it follows from the $J$-linearity of $\nabla$ that 
\begin{equation}\label{jcon1}
\nabla_{\partial_z}\partial_{\bar z}=0 
\end{equation}
and furthermore that there exists a unique complex-valued function $\gamma$ on $U$ such that
\begin{equation}\label{jcon2}
\nabla_{\partial_z}\partial_z=\gamma\,\partial_z.
\end{equation}
Conversely, a torsion-free connection satisfying~\eqref{jcon1} and~\eqref{jcon2} in every holomorphic coordinate system is $J$-linear and hence a Weyl connection. 

Suppose $\sigma : U \to L^{\ell}$ is a smooth section, we write $\sigma=s\,\d z^{\ell}$ for some smooth complex-valued function $s$ on $U$. Then we have
\begin{equation}\label{defopcoord}
\nabla^{1,0}s\, \d z^{\ell}=\left(\partial_z s-\ell\gamma s\right)\d z^{\ell+1}.
\end{equation}
If $w: V \to \C$ is another local holomorphic coordinate with $U\cap V\neq \emptyset$, then writing $\sigma=\hat{s} \d w^{\ell}$ and $\nabla_{\partial_w}\partial_w=\hat{\gamma}\,\partial_w$ for complex-valued functions $\hat{s},\hat{\gamma}$ on $V$, we have 
\begin{equation}\label{changecoord}
\hat{s}=(\partial_wz)^{\ell}s,\quad \hat{\gamma}=\left(\partial_w z\right)\gamma+\left(\partial_z w\right)\partial^2_{ww}z,
\end{equation}
on $U\cap V$. Using~\eqref{changecoord} it is easy to check that the right hand side of~\eqref{defopcoord} does not depend on the chosen coordinates. The reader less familiar with complex geometry may therefore take~\eqref{defopcoord} as the definition of the operators $\nabla^{1,0}$. Note in particular that $\nabla^{1,0}$ agrees with the usual ``del-operator'' $\partial$ on functions.

\subsection{Higher order Schwarzian derivatives}

Let $(M,[g])$ be a Riemann surface and $\nabla$ a Weyl connection for $[g]$. Let $f$ be a holomorphic function on $M$ satisfying $f^{\prime}=\partial f\neq 0$. We define the Schwarzian derivative of $f$ with respect to $\nabla$ to be the quadratic differential 
$$
\mathrm{S}^{\nabla}(f)=\frac{f^{\prime\prime\prime}}{f^{\prime}}-\frac{3}{2}\left(\frac{f^{\prime\prime}}{f^{\prime}}\right)^2.
$$
In a local holomorphic coordinate system $z : U \to \C$ with $\nabla_{\partial_z}\partial_z=\gamma\, \partial_z$ for some complex-valued function $\gamma$ on $U$, we obtain with~\eqref{defopcoord}
\begin{equation}\label{schwarzdiffcoord}
\mathrm{S}^{\nabla}(f)=\left\{f,z\right\}\d z^2+\left(\frac{1}{2}\gamma^2-\partial_z \gamma\right)\d z^2, 
\end{equation}
where
$$
\left\{f,z\right\}=\frac{\partial^3_{zzz}f}{\partial_z f}-\left(\frac{\partial^2_{zz}f}{\partial_z f}\right)^2
$$
denotes the classical Schwarzian derivative of $f$ with respect to the coordinate $z$. Taking the Levi-Civita connection of the spherical metric 
$$
g_1=\left(\frac{2}{1+|z|^2}\right)^2\d z \circ \d \bar z
$$ 
on the Riemann sphere $\hat{\C}$ gives
$$
{}^{g_1}\nabla_{\partial_z}\partial_z=-\frac{2\bar z}{1+|z|^2}\,\partial_z. 
$$
Consequently, we have $\gamma=-2\bar{z}/(1+|z|^2)$, so that
$$
\frac{1}{2}\gamma^2-\partial_z \gamma=0,
$$
and~\eqref{schwarzdiffcoord} simplifies to give the classical coordinate definition of the Schwarzian derivative. Recall that the classical definition of the Schwarzian derivative extends to the set of locally injective meromorphic functions and consequently by using~\eqref{schwarzdiffcoord}, so does our definition. Furthermore, a locally injective meromorphic function $a$ defined on some domain $\Omega\subset \hat{\C}$ satisfies
\begin{equation}\label{moebvanish}
\mathrm{S}^{{}^{g_1}\nabla}(a)=0
\end{equation}
if and only if $a$ is the restriction of a M\"obius transformation. A crucial property of the Schwarzian derivative is that it defines a cocycle. If $f$ is a locally injective meromorphic function on $M$ and $a$ a local biholomorphism on $\hat{\C}$ so that $a \circ f$ is well defined, then one may easily verify that
\begin{equation}\label{cocycle}
\mathrm{S}^{\nabla}(a\circ f)=\mathrm{S}^{\nabla}(f)+f^*\left(\mathrm{S}^{{}^{g_1}\nabla}(a)\right).
\end{equation}
As a consequence of this identity and~\eqref{moebvanish} we see that the Schwarzian derivative is invariant under post-composition by a M\"obius transformation. 

The Schwarzian derivative is the first in a sequence of differential operators enjoying invariance under post-composition by a M\"obius transformation: 
\begin{defn}
Let $(M,[g])$ be a Riemann surface equipped with a Weyl connection $\nabla$. For $\ell\geqslant 2$ we define the \textit{$\ell$-th Schwarzian derivative} to be the $(\ell+1)$-th order differential operator defined by
$$
\mathrm{S}^{\nabla}_{\ell+1}(f)=\left(\mathrm{S}^{\nabla}_{\ell}(f)\right)^{\prime}\quad \text{where}\quad \mathrm{S}^{\nabla}_2(f)=\frac{f^{\prime\prime\prime}}{f^{\prime}}-\frac{3}{2}\left(\frac{f^{\prime\prime}}{f^{\prime}}\right)^2.
$$
The $\ell$-th Schwarzian derivative maps a locally injective meromorphic function $f$ on $M$ to a smooth section of $L^{\ell}$, that is, a smooth differential on $M$ of degree $\ell$.  
\end{defn}
\begin{rmk}
By construction, the operators $\mathrm{S}^{\nabla}_{\ell}$ are invariant under post-com\-po\-si\-tion by a M\"obius transformation. However, none of the higher order Schwarzian derivatives defines a cocycle. Local coordinate definitions of higher order M\"obius invariant Schwarzian differential operators previously appeared in~\cite{MR0249622} and~\cite{MR1386108}. Notice that the higher order Schwarzian derivatives introduced in \cite{MR1386108} differ from the ones introduced here, see in particular \cite[Equation 1-19 in Theorem D]{MR1386108}. 
\end{rmk}
Every non-vanishing holomorphic quadratic differential $Q$ on a Riemann surface $(M,[g])$ gives rise to a flat $[g]$-conformal connection ${}^A\nabla$ where ${}^A\nabla$ denotes the Levi-Civita con\-nection of the flat Lorentzian metric $A=\mathrm{Re}(Q)$. Indeed, in a local holomorphic coordinate $z : U \to \C$ on $M$ so that $Q=q \d z^2$ for some holomorphic function $q$ on $U$, we have ${}^A\nabla_{\partial_z}\partial_{\bar z}=0$ and
$$
{}^A\nabla_{\partial_z}\partial_{z}=\gamma\,\partial_z,
$$
with
\begin{equation}\label{eqforgamma}
\gamma=\frac{\partial_z q}{2q},
\end{equation}
showing that ${}^A\nabla$ is $[g]$-conformal. If $f$ is a locally injective meromorphic function on $U$, we obtain with~\eqref{defopcoord}
\begin{equation}\label{schwarzwithq}
\mathrm{S}^{{}^A\nabla}_2(f)=\left(\left\{f,z\right\}+\frac{5}{8}\left(\frac{\partial_z q}{q}\right)^2-\frac{1}{2}\frac{\partial^2_{zz}q}{q}\right)\d z^2.
\end{equation} 

\section{The entropy differentials and their invariance properties}

\subsection{The entropy differentials}

Let $M$ be an oriented smooth surface and $\bx=(\bx^i) : M \to \mathbb{E}^3$ a smooth immersion into Euclidean $3$-space. Let $g$ and $A$ denote the induced metric and second fundamental form on $M$
$$
g=\d \bx \cdot \d\bx, \quad A=-\d \mathbf{N}\cdot \d\bx,
$$
where $\mathbf{N}=(\mathbf{N}^i) : M \to S^2\subset \mathbb{E}^3$ denotes the (orientation compatible) Gauss map of $\bx$. The Weingarten shape operator is the endomorphism $S : TM \to TM$ satisfying
$$
A(X,Y)=g(S(X),Y)=g(X,S(Y))
$$ 
for all $X,Y \in \Gamma(TM)$. The operator $S$ is $g$-symmetric and hence (pointwise) diagonalizable with real eigenvalues $\kappa_1,\kappa_2$, called the principal curvatures of $\bx$. Recall that $\bx$ is called \textit{minimal} if its mean curvature $H=\frac{1}{2}\tr S=\frac{1}{2}(\kappa_1+\kappa_2)$ vanishes identically. A point $p \in M$ where $\kappa_1=\kappa_2$ is an umbilic point, which -- in the minimal case -- amounts to $\kappa_1=\kappa_2=0$. Consequently, the umbilic points are precisely those points where the Gauss curvature $K=\kappa_1\kappa_2=-(\kappa_1)^2$ vanishes. The umbilic locus is the set
$$
\mathcal{U}=\left\{p \in M \,|\, \kappa_1(p)=\kappa_2(p)=0\right\}
$$
and we use $M^{\prime}$ to denote its complement in $M$, i.e., $M^{\prime}=M\setminus \mathcal{U}$. The minimality of $\bx$ is well-known to be equivalent to the meromorphicity of the stereographically projected Gauss-map. More precisely, we have the following lemma whose proof may be found in~\cite{MR552581} or most standard texts on minimal surfaces.
\begin{lem}\label{gaussholo}
Let $\bx : M \to \E^3$ be a smooth immersion with stereographically projected Gauss map $G=(\mathbf{N}^1+\i\mathbf{N}^2)/(1-\mathbf{N}^3) : M \to \hat{\C}$. Then $\bx$ is minimal if and only if $G$ is meromorphic with respect to the complex structure on $M$ induced by $g$ and the orientation. 
\end{lem}
Another key property of minimal surfaces is that the induced metric satisfies the so-called Ricci condition. 
\begin{thm}
The induced metric $g$ of a minimal immersion $\bx : M \to \E^3$ has non-positive Gauss curvature $K$ and whenever $K<0$, the metric $g$ satisfies 
\begin{equation}\label{riccicond}
\Delta \log(-K)=4K. 
\end{equation}
Conversely, if $(M,g)$ is a simply connected Riemannian $2$-manifold of strictly negative Gauss curvature satisfying~\eqref{riccicond}, then $(M,g)$ can be immersed isometrically and minimally into $\E^3$.  
\end{thm}
\begin{rmk}
Here $\Delta$ denotes the Laplace-Beltrami operator with respect to $g$. For a proof the reader may consult~\cite{MR655419}.
\end{rmk}
An immediate consequence of~\eqref{riccicond} is the following result. 
\begin{cor}\label{flatmetric}
Let $g$ denote the induced metric of a minimal immersion $\bx : M \to \E^3$ and $K$ its Gauss curvature. Then on $M^{\prime}\subset M$, the complement of the umbilic locus, $g_0=\sqrt{-K}g$ defines a flat metric. 
\end{cor}
\begin{proof}The proof is an immediate consequence of the well-known and easily-derived formula for the dependence of the Gauss curvature on a conformal change of the metric. For $u \in C^{\infty}(M)$, we have
$$
K_{e^{2u}g}=\e^{-2u}\left(K_g-\Delta u\right).
$$
Consequently, on $M^{\prime}$ we obtain
$$
K_{g_0}=\frac{1}{\sqrt{-K}}\left(K-\frac{1}{2}\Delta\log\left(\sqrt{-K}\right)\right)=\frac{1}{\sqrt{-K}}\left(K-\frac{1}{4}\Delta\log\left(-K\right)\right)=0.
$$
\end{proof}
Clearly, the Levi-Civita connection ${}^{g_0}\nabla$ of $g_0$ is a $[g]$-conformal connection. Combining \cref{gaussholo} and \cref{flatmetric} we immediately see that we obtain a sequence of differentials $P_{\ell}$ on the complement $M^{\prime}$ of the umbilic locus of a minimal immersion $\bx : M \to \E^3$. Indeed, since the Gauss map is a locally injective meromorphic function on $M^{\prime}$, the differentials 
$$
P_{\ell}=\mathrm{S}^{{}^{g_0}\nabla}_{\ell}(G) 
$$
are well defined on $M^{\prime}$. 

%\begin{rmk}
%The differentials $\seq$ admit an interpretation as higher order conservation laws for the Liouville equation $4\partial_{z\bar z}u=\exp(-2u)$ which locally describes minimal surfaces in $\E^3$. For a comprehensive study of the conservation laws of the sinh-Gordon and cosh-Gordon equation -- which locally describe constant mean curvature surfaces in $3$-dimensional space forms -- we refer the reader to~\cite{arxiv:1309.6606}.\footnote{More precisely, surfaces of constant mean curvature $-2\delta$ in a three-dimensional space form of constant sectional curvature $\kappa$ locally correspond to the solutions of the sinh-Gordon equation provided $\kappa+\delta^2>0$ and to the cosh-Gordan equation provided $\kappa+\delta^2<0$.} 
%\end{rmk}

On $M^{\prime}$ we have another flat metric given by the second fundamental form. 
\begin{lem}\label{levihopf=leviflat}
On $M^{\prime}$ the Levi-Civita connection ${}^{g_0}\nabla$ of the flat Riemannian metric $g_0$ and the Levi-Civita connection ${}^A\nabla$ of the flat Lorentzian metric $A$ agree.  
\end{lem}
\begin{proof}
Denote by $J$ the complex structure on $M$ induced by $g$ and the orientation. Then the Hopf differential $Q=A+\i AJ$ is a holomorphic quadratic differential~\cite{MR1013786}, where we write
$$
AJ(X,Y)=A(JX,Y),
$$
for $X,Y \in \Gamma(TM)$. In a neighbourhood of every point $p \in M^{\prime}$ we may find local holomorphic coordinates $z=x+\i y$ such that $Q=\d z^2$. We will henceforth call such coordinates $Q$-adapted. In such a coordinate system we have
$$
g=\e^{2u}\,\d z \circ {\d \bar z}=\e^{2u}\left(\d x^2+\d y^2\right), \quad \text{and}\quad A=\mathrm{Re}(\d z^2)=\d x^2-\d y^2. 
$$
for some real-valued function $u$ satisfying Liouville's equation 
\begin{equation}\label{liouville}
4\partial^{2}_{z\bar z} u=\e^{-2u}.
\end{equation}
It follows that $g$ has Gauss curvature $K=-\e^{-4u}$ and hence 
$$
g_0=\sqrt{-K}g=\d x^2+\d y^2. 
$$
Therefore, with respect to the coordinate $z$, the Christoffel symbols of both ${}^{g_0}\nabla$ and ${}^A\nabla$ vanish identically. Covering $M^{\prime}$ with local holomorphic $Q$-adapted coordinates implies that ${}^{g_0}\nabla={}^A\nabla$. 
\end{proof}
The differentials $P_{\ell}$ have the property of being holomorphic away from umbilic points and extending meromorphically across umbilic points. 
\begin{ppn}\label{mainthmp1}
Let $\bx : M\to \E^3$ be a non-flat minimal immersion with induced differentials $\seq$ on the complement $M^{\prime}$ of the umbilic locus. Then the differentials $\seq$ are holomorphic on $M^{\prime}$ and extend meromorphically to all of M with $P_{\ell}$ having a pole of order $\ell$ at the umbilic points.  Furthermore, the residue of $P_{\ell}$ at an umbilic point $p \in M$ is
$$
\mathrm{Res}_p(P_{\ell})=\left(-\frac{1}{2}\right)^{\ell+1}(\ell-1)!\,(n+2)^{\ell-2}(3n^2+4n),
$$
where $n$ denotes the order of vanishing of the Hopf differential at $p$.   
\end{ppn}
\begin{rmk}
Suppose that $P$ is a meromorphic differential of degree $\ell \in \mathbb{N}$ having a pole of order $\ell$ at some point $p$, so that there exists a local $p$-centred holomorphic coordinate $z$ satisfying
$$
P(z)=\frac{f(z)}{z^{\ell}}\d z^{\ell},
$$
where $f$ is a holomorphic function near $0$ with $f(0)\neq 0$. Then the~\textit{residue} of $P$ at $p$ is defined as
$$
\mathrm{Res}_{p}(P)=f(0), 
$$
which is clearly well-defined, that is, independent of the choice of local $p$-centred holo\-mor\-phic coordinate. 
\end{rmk}
\begin{proof}[Proof of \cref{mainthmp1}]
Let $p \in M^{\prime}$ and let $w : U_p \to \C$ be a local holomorphic coordinate defined in a neighbourhood of $p$ so that $Q=\d w^2$. Therefore, in such a coordinate system~\eqref{schwarzdiffcoord} becomes
\begin{equation}\label{someq}
P_2=\left\{G,w\right\} \d w^2
\end{equation}
and 
\begin{equation}\label{someid15}
P_{\ell+1}=\partial_wp_{\ell}\d w^{\ell+1}
\end{equation} 
where we write $P_{\ell}=p_{\ell}\d w^{\ell}$. Since the Schwarzian derivative maps meromorphic locally injective functions to holomorphic quadratic differentials, it follows that $P_2$ and hence the differentials $P_{\ell}$ on $M^{\prime}$ are holomorphic. 

Now suppose $z$ is any local holomorphic coordinate defined in some $p$-neighbour\-hood $V_p$ so that $Q=q\d z^2$ for some holomorphic function $q$ on $V_p$. Note that the umbilic points are isolated since they are precisely the points where the holomorphic quadratic differential $Q$ vanishes. Hence, at an umbilic point $p \in M$ we may choose a local $p$-centred holomorphic coordinate $z$ so that $\bx^3=\mathrm{Re}(z)$ and  
$$
G=1+az^{n+1}+bz^{n+2}+O(z^{n+2}),
$$
where $a \in \C^*$ and $b \in \C$. From this we compute that
$$
Q=-a(n+1)z^n\d z^2+O(z^n),
$$
so that~\eqref{schwarzwithq} yields
\begin{equation}\label{polecoord}
P_2=-\left(\left(\frac{3n^2+4n}{8}\right)z^{-2}+\frac{n(n+2)}{(n+1)}\frac{b}{a}z^{-1}\right)\d z^2+O(1).
\end{equation}
Hence we may write 
$$
P_2=\frac{f_{2,n}(z)}{z^2}\d z^2,
$$
where the complex-valued function $f_{2,n}$ is holomorphic near $0$ and satisfies
$$
f_{2,n}(0)=-\frac{3n^2+4n}{8}.
$$
%which can be brought into the form
%\begin{equation}\label{polecoord}
%P_2=-\left(\frac{3n^2+4n}{8}\right)\frac{\d z^2}{z^2}+O(1),
%\end{equation}
%by a change of coordinate $z\mapsto z+cz^2$ for some appropriate constant $c$. 
Using~\eqref{someid15} and~\eqref{polecoord} together with a straightforward inductive argument we see that  
$$
P_{\ell}=\left(\sum_{k=0}^{\ell-1}c_{\ell,k}\,z^{k-\ell}\left(\frac{b}{a}\right)^{k}\right)\d z^{\ell}+O(1)
$$
for some coefficients $c_{\ell,k}$ where
$$
c_{\ell,0}=\left(-\frac{1}{2}\right)^{\ell+1}(\ell-1)!\,n(n+2)^{\ell-2}(3n+4). 
$$
It follows that we may write
$$
P_{\ell}=\frac{f_{\ell,n}(z)}{z^{\ell}}\d z^{\ell}
$$
where the complex-valued function $f_{\ell,n}$ is holomorphic near $0$ and satisfies
$$
\mathrm{Res}_p(P_{\ell})=f_{\ell,n}(0)=c_{\ell,0}, 
$$
thus completing the proof.
%A coordinate transformation of the form $z\mapsto z+\sum_{k=0}^{\ell-2} d_{k}z^{k+2}$ for some appropriate constants $d_k$ then yields the form
%\begin{equation}\label{polecoordsgeneral}
%P_{\ell}=\left[\left(-\frac{1}{2}\right)^{\ell+1}(\ell-1)!\,n(n+2)^{\ell-2}(3n+4)z^{-\ell}\right]\d z^{\ell}+O(1),
%\end{equation}
%thus proving the claim. 
\end{proof}
\begin{rmk}
In~\cite{MR3667429} it was shown that the real part of $P_2$ is a conservation law for critical points of a certain curvature entropy functional. For this reason we call $\seq$ the \textit{entropy differentials}. In fact, it was shown that if we write $g=\exp(2u)\d z\circ \d \bar z$ for some real-valued function $u$ and some local $Q$-adapted holomorphic coordinate $z$ so that $Q=\d z^2$, then $P_2$ is given by
\begin{equation}\label{entropycoord}
P_2=-2\left(\partial^2_{zz} u+\left(\partial_z u\right)^2\right)\d z^2. 
\end{equation}
\end{rmk}
\subsection{Invariance properties} The entropy sequence enjoys certain invariance properties which we will now discuss. First we observe (see also~\cite{MR3667429}):
\begin{ppn}\label{mainthmp2}
The entropy sequence is intrinsic, i.e.~$P_{\ell+2}$ depends on the induced metric and its derivatives up to order $\ell+4$ only. In particular, $P_{\ell+2}$ is invariant under post-composing $\bx : M \to \E^3$ with an ambient isometry.
\end{ppn}
\begin{proof}
We work in a local $Q$-adapted local coordinate $z$ so that $Q=\d z^2$ and $g=\e^{2u}\d z \circ \d \bar z$. Consider the metric $\hat{g}=(-K)^{3/4}g$ with Gauss curvature
$$
K_{\hat{g}}=\frac{1}{2}|K_g|^{1/4}
$$
where $K_g$ denotes the Gauss curvature of $g$. Now define
$$
T={}^{\hat{g}}\mathring{\nabla}^2\log(K_{\hat{g}})=\frac{1}{4}{}^{\hat{g}}\mathring{\nabla}^2\log(-K_g)
$$
where ${}^{\hat{g}}\mathring{\nabla}^2$ denotes the trace-free Hessian of $\hat{g}$. Clearly, the symmetric traceless covariant $2$-tensor field $T$ depends on $g$ and its derivatives up to order four. Using the identity $u=-\frac{1}{4}\log(-K_g)$, we compute
$$
\aligned
T&=-{}^{g_0}\mathring{\nabla}^2u-\d u^2 +\frac{1}{2}g_0({}^{g_0}\nabla u,{}^{g_0}\nabla u)g_0\\
&=-2\operatorname{Re}\left(\left(\partial^2_{zz}u+(\partial_z u)^2\right)\d z^2\right),
\endaligned
$$
which agrees with the real part of~\eqref{entropycoord}. It follows that $P_2$ depends on the induced metric only. Since $P_{\ell+2}$ is just the $\ell$-th derivative of $P_2$ with respect to the Levi-Civita connection of the induced flat metric $g_0$, we see that $P_{\ell+2}$ depends on the induced metric and its derivatives up to order $\ell+4$ only. 
\end{proof} 
In order to discuss the further invariance properties we first recall the Weierstrass representation of minimal surfaces. We consider $\C^3$ and let $\phi$ denote the natural complex inner product 
$$
\phi(\mathbf{z},\mathbf{w})=z_1w_1+z_2w_2+z_3w_3,
$$
where $\mathbf{z}=(z_i)$ and $\mathbf{w}=(w_i)$ are elements of $\C^3$. Weierstrass observed that if $(M,J)$ is a Riemann surface and $\tilde{\mathbf{X}} : M \to \C^3$ is a holomorphic null immersion, i.e.~$\tilde{\mathbf{X}}^*\phi=0$, then $\mathbf{X}=\mathrm{Re}\circ \tilde{\mathbf{X}} : M \to \R^3$ is a conformal and minimal immersion of $(M,J)$. Conversely, every minimal immersion of a simply connected surface $M$ arises in this way. 

Having a holomorphic null immersion $\tilde{\bx} : M \to \C^3$, the triple $(M,G,\eta)$, where $G$ is the stereographically projected Gauss map of the minimal immersion $\mathrm{Re}(\tilde{\bx})$ and $\eta=\d \tilde{\bx}^3$ is called the ~\textit{Weierstrass data} of the null immersion $\tilde{\bx}$. Standard computations give (see for instance~\cite{karcher}) 
\begin{equation}\label{relheighthopf}
Q=-\frac{1}{G}\d G \circ \eta,
\end{equation}
where $Q$ denotes the Hopf differential of $\mathrm{Re}(\tilde{\bx})$. 

A (global) method of producing a holomorphic null immersion of a Riemann surface $(M,J)$ into $\C^3$ from Weierstrass data was obtained by Osserman~\cite{MR0179701}:
\begin{thm}\label{osswei}
Let $G : M \to \hat{\C}$ be a meromorphic function and $\eta$ a holomorphic $1$-form on $M$. Suppose that:
\begin{itemize}
\item[(i)] The zeroes of $\eta$ coincide with the poles and zeros of $G$, with the same order;
\item[(ii)] For any closed curve $\gamma\subset M$,
$$
\overline{\int_{\gamma}G\eta} = \int_{\gamma}\frac{\eta}{G}, \quad \mathrm{Re}\int_{\gamma}\eta=0, 
$$
where $\bar{z}$ denotes complex conjugation of $z \in \C$; 
\end{itemize}
then
$$
\tilde{\bx}(p)-\tilde{\bx}(p_0)=\int_{p_0}^p\left(\frac{1}{2}(G^{-1}-G),\frac{\i}{2}(G^{-1}+G),1\right)\eta,
$$
yields a holomorphic null immersion $\tilde{\bx} : M \to \C^3$ so that $\Re(\tilde{\bx})$ has Weierstrass data $(M,G,\eta)$.
\end{thm}

The natural left action of the linear conformal group $\C^*\times \mathrm{SO}(3,\C)$ on $\C^3$ yields a left action on the space of holomorphic null immersions and consequently on the space $\mathcal{W}(M)$ of minimal immersions of $M$ arising via the Weierstrass representation. For an element $\bx : M \to \E^3 \in \mathcal{W}(M)$ we call the deformations of $\bx$ obtained by the $\C^*\times \mathrm{SO}(3,\C)$ action its~\textit{Weierstrass deformations}. Next, following~\cite{arxiv:dg-ga/9512003}, we study the space of Weierstrass deformations more carefully. 

The action by the subgroup $\R^+\times \mathrm{SO}(3,\R)$ corresponds to similarity transformations of the minimal surface associated to the null immersion. It follows that the space of non-similar Weierstrass deformations is the homogeneous space
$$
\left(\C^*\times \mathrm{SO}(3,\C)\right)/\left(\R^+\times \mathrm{SO}(3,\R)\right)\simeq S^1\times \left(\mathrm{SO}(3,\C)/\mathrm{SO}(3,\R)\right).
$$
The first circle factor yields the well-known \textit{associated family} (or \textit{Bonnet family}) of minimal surfaces. The associated family of minimal surfaces has the properties of being locally isometric and sharing a common Gauss map. Furthermore, the Hopf differential $Q$ changes by a complex phase. The latter factor gives rise to the so-called \textit{Goursat family} of minimal surfaces. The group $\mathrm{SO}(3,\C)\simeq \mathrm{PSL}(2,\C)$ acts on the Gauss-map by M\"obius transformation and leaves the Hopf differential unchanged (c.f.~\cite[Lemma 5.3.1]{MR2004958} or~\cite{arxiv:dg-ga/9512003}). 

Since under a Bonnet-transform the induced metric is unchanged, so is the induced flat metric. It follows that the differentials $\seq$ are the same for the whole $S^1$-family of associated minimal surfaces. Furthermore, since the Hopf differential is unchanged under a Goursat transform, so is the Levi-Civita connection of the second fundamental form ${}^A\nabla$. \cref{levihopf=leviflat} implies that the Levi-Civita connection ${}^{g_0}\nabla$ of the flat metric is invariant under Goursat transform as well (this is  noteworthy since the induced metric itself does change non-trivially under Goursat transforms). Hence it follows from the invariance of the Schwarzian derivative under post-composition by a M\"obius transformation that $P_2$ and hence all differentials $\seq$ are invariant under Goursat transforms. Concluding, we have:
\begin{ppn}\label{mainthmp3}
The meromorphic differentials $\seq$ are invariant under Gour\-sat -- and Bonnet transforms. 
\end{ppn} 

Finally, note that scaling the immersion $\bx : M \to \E^3$ by a constant does neither change the Levi-Civita connection of the induced flat metric, nor the Gauss map and hence leaves the sequence $\seq$ unchanged. This fact together with the content of \cref{mainthmp1}, \cref{mainthmp2} and \cref{mainthmp3} is summarized in \cref{mainthma} of the introduction. 

\section{The degree of a minimal surface}
The existence of an intrinsic sequence of meromorphic differentials on a minimal surface motivates the following definition:
\begin{defn}
Let $M$ be an oriented smooth surface. A non-flat minimal immersion $\bx : M \to \E^3$ is said to have \textit{degree} $n \in \mathbb{N}$ if $P_n$ vanishes identically or if there exists a weighted-homogeneous polynomial $\psi_n : \C^{n-2} \to \C$ of degree $n$ with weights $(2,3,\ldots,n-1)$ such that
\begin{equation}\label{degreecond}
P_n=\psi_n(P_2,\ldots P_{n-1}).
\end{equation}
Furthermore, we call $\psi_n$ the \textit{algebraic type} of the immersion $\bx$. Clearly, if a non-flat minimal immersion has degree $n$, then it also has degree $m$ for all $m\geqslant n$. The degree will therefore always denote the smallest integer for which a relation of the form~\eqref{degreecond} holds. 
\end{defn}
\begin{rmk}
Recall, polynomial $\psi_n(z_1,z_2,\ldots,z_{n-1})$ which does not vanish identically is called \textit{weighted-ho\-mo\-ge\-ne\-ous of degree $n$} if there exist positive integers $(w_1,\ldots,w_{n-1})$, called the \textit{weights of the variables}, such that for every $\lambda \neq 0$ 
$$
\psi_n(\lambda^{w_1}z_1,\lambda^{w_2}z_2,\ldots,\lambda^{w_{n-1}}z_{n-1})=\lambda^n\psi_n(z_1,z_2,\ldots,z_{n-1}).
$$
\end{rmk}
\begin{rmk}
A definition of degree for constant mean curvature surfaces without umbilic points was previously given by Pinkall and Sterling~\cite{MR1014929}.
\end{rmk}

\begin{ex}[Enneper's surface]
Enneper's surface is the minimal surface with Weierstrass data $M=\C$, $G(z)=z$ and $\eta(z)=z \d z$. From~\eqref{relheighthopf} we compute
$$
Q(z)=-\d z^2
$$
and hence using~\eqref{schwarzwithq} we obtain
$$
P_2(z)=\left\{z,z\right\}\d z^2=0. 
$$
Therefore, Enneper's surface has the lowest possible degree two. Conversely, it was shown in~\cite{MR3667429} that a minimal surface satisfying $P_2=0$ is -- up to Euclidean motion and scaling -- an open subset of Enneper's surface. Thus Enneper's surface and its Weierstrass deformations exhaust all degree $2$ surfaces. Note that Enneper's Weierstrass deformations are just similarity transforms.\end{ex}

\begin{rmk}
A characterisation of Enneper's surface in terms of so-called Chern--Ricci functions was given in \cite{MR3885861}. 
\end{rmk}

\begin{ex}[The helicoid \& catenoid] The helicoid is the minimal surface with Weierstrass data $M=\C$, $G(z)=\e^z$ and $\eta(z)=\i \d z$. From~\eqref{relheighthopf} we compute
$$
Q(z)=-\i \d z^2
$$
and hence using~\eqref{schwarzwithq} we obtain
$$
P_2(z)=\left\{\e^z,z\right\}\d z^2=-\frac{1}{2}\d z^2,
$$
so that
$$
P_3(z)=\partial_z\left(-\frac{1}{2}\right)\d z^3=0, 
$$
showing that the helicoid is a degree $3$ surface (and consequently so is the catenoid, since it is a member of the helicoid's associated family of minimal surfaces). Conversely, it was proved in~\cite{MR3667429} that a minimal surface satisfying $P_2=\lambda Q$ for some non-zero complex constant $\lambda$ is -- up to Euclidean motion, scaling and Goursat transform -- an open subset of a member of the helicoid/catenoid family. Since $P_2=\lambda Q$ precisely characterizes the degree $3$ surfaces, it follows that the helicoid and its Weierstrass deformations exhaust all degree $3$ surfaces.  
\end{ex}

\subsection{Degree four surfaces}
A minimal immersion with entropy differentials $\seq$ is of degree $4$ if and only if there exists a complex constant $\lambda$ such that
\begin{equation}\label{degreefourcond}
P_4+\lambda (P_2)^2=0.
\end{equation}
Therefore, we have a (complex) one-dimensional space of algebraic types of degree $4$ immersions and a complete classification is not feasible. However, the space of algebraic types of degree four surfaces simplifies in the presence of umbilic points. 
\begin{ppn}\label{polyconstraint}
Suppose $\bx : M\to \mathbb{E}^3$ is a smooth minimal immersion of degree four whose Hopf differential vanishes to order $n\in\mathbb{N}$ at some umbilic point. Then the Hopf differential vanishes to order $n$ at every umbilic point and the immersion $\bx$ satisfies
$$
P_4+\frac{12(n+2)^2}{(3n^2+4n)}(P_2)^2=0. 
$$  
\end{ppn}
\begin{proof}
Let $p \in M$ denote the umbilic point at which the Hopf differential vanishes to order $n$. Then \cref{mainthmp1} implies that
$$
\mathrm{Res}_p(P_2)=-\frac{3n^2+4n}{8}. 
$$
and hence
$$
\mathrm{Res}_p((P_2)^2)=\left(\mathrm{Res}_p(P_2)\right)^2=\left(\frac{3n^2+4n}{8}\right)^2.
$$
Since $\bx$ has degree four there exists a complex-constant $\lambda$ such that $P_4-6\lambda(P_2)^2=0$ and hence using \cref{mainthmp1} we must have
$$
\mathrm{Res}_p(P_4)=-\frac{3}{16}(n+2)^2(3n^2+4n)=6\lambda \left(\mathrm{Res}_p(P_2)\right)^2=6\lambda\left(\frac{3n^2+4n}{8}\right)^2,
$$
which implies
$$
P_4+\frac{12(n+2)^2}{(3n^2+4n)}(P_2)^2=0,
$$
as claimed. Furthermore, note that the sequence $a_n=\frac{12(n+2)^2}{(3n^2+4n)}$ is injective which excludes the possibility of having two umbilic points at which the Hopf differential vanishes to different order. 
\end{proof}
\begin{ex}[Enneper surfaces of higher dihedral symmetry] The minimal surfaces arising from the Weierstrass data $M=\C$, $G(z)=z^k$ and $\eta(z)=z^k\d z^k$ for some integer $k\geqslant 2$ are called Enneper surfaces of higher dihedral symmetry. From~\eqref{relheighthopf} we obtain 
\begin{equation}\label{hopfgenenep}
Q(z)=-kz^{k-1}\d z^2,
\end{equation}
showing that these surfaces have an umbilic point at which the Hopf differential vanishes to order $k-1$. Therefore, by \cref{polyconstraint}, if the Enneper surfaces of higher dihedral symmetry have degree four, then they will have to satisfy
$$
P_4+12\frac{(k+1)^2}{(k-1)(3k+1)}(P_2)^2=0.
$$
From~\eqref{eqforgamma} and~\eqref{hopfgenenep} we get $\gamma=\frac{(k-1)}{2z}$, hence with~\eqref{schwarzwithq} we obtain
$$
\aligned
P_2(z)&=\left(\left\{z^k,z\right\}-\frac{(k-1)(k+3)}{8z^2}\right)\d z^2\\
&=-\left(\frac{(k-1)(k+1)}{2z^2}+\frac{(k-1)(k+3)}{8z^2}\right)\d z^2\\
&=-\frac{1}{8z^2}(k-1)(3k+1)\d z^2.
\endaligned
$$
Therefore, using~\eqref{defopcoord} we have
$$
\aligned
P_3(z)&=\left(-\frac{1}{8}(k-1)(3k+1)\partial_z\left(\frac{1}{z^2}\right)+\frac{1}{4}\frac{(k-1)}{2z}\left(\frac{(k-1)(3k+1)}{z^2}\right)\right)\d z^3\\
&=\frac{1}{8z^3}(k-1)(k+1)(3k+1)\d z^3,
\endaligned
$$
and likewise
$$
\aligned
P_4(z)&=\left(-\frac{3}{8z^4}(k-1)(k+1)(3k+1)\right.\\
&\phantom{=}\left.-\frac{3}{2z}(k-1)\frac{1}{8z^3}(k-1)(k+1)(3k+1)\right)\d z^4\\
&=-\frac{3}{16z^4}(k+1)^2(k-1)(3k+1)\d z^4.
\endaligned
$$
Hence we obtain
$$
P_4+12\frac{(k+1)^2}{(k-1)(3k+1)}(P_2)^2=0,
$$
which agrees with \cref{polyconstraint}. 
\end{ex}
The sequence 
$$
a_n=\frac{12(n+2)^2}{(3n^2+4n)}
$$
describing the admissible coefficients for degree four minimal immersions with umbilic points satisfies $\lim_{n \to \infty}\lambda_n=4$. It is therefore natural to ask what (umbilic-free) minimal surfaces satisfy $P_4+4(P_2)^2=0$. We will next study a $1$-parameter family of such surfaces.  
\begin{ex}[Limit degree four surfaces]
Let $M=\C$ with $G_t(z)=\frac{1}{t}\e^{-z}$ and $\eta_t(z)=2t\e^{z}\d z$ where $t>0$ is a real parameter. From this Weierstrass data we obtain a family of minimal immersions $\bx_t : \R^2 \to \R^3$ given by
$$
(x,y) \mapsto \left(x-\frac{t^2}{2}\mathrm{e}^{2x}\cos(2y),y+\frac{t^2}{2}\mathrm{e}^{2x}\sin(2y),-2t\mathrm{e}^x\cos(y)\right)
$$
where $z=x+\i y$ denotes the standard coordinate on $\C\simeq \R^2$. Note that the immersion $\bx_0$ yields the plane. The induced metrics are
$$
g_t=\left(1+t^2\mathrm{e}^{2x}\right)^2\left(\d x^2+\ \d y^2\right),
$$
and the second fundamental forms are 
$$
A_t=2t\mathrm{e}^x\left(\cos(y)\left(\d x^2-\d y^2\right)-2\sin(y)\left(\d x \circ \d y\right)\right).
$$
The metrics $g_t$ have Gauss curvature
$$
K_t=-\frac{4t^2\mathrm{e}^{2x}}{\left(1+t^2\mathrm{e}^{2x}\right)^4}
$$
Writing $\Sigma_t=\bx_t(\C)$, these surfaces are singly periodic with respect to the action  $\Sigma_t \times \mathbb{Z} \to \Sigma_t$ 
$$
(p,n)\mapsto p+\left(\begin{array}{c} 0 \\ 2n\pi \\ 0 \end{array}\right). 
$$
for $n \in \mathbb{Z}$ and $p \in \Sigma$. The fundamental domain $\bx_t(\R \times [-\pi,\pi])$ is symmetric around $y=0$ and contains two straight lines defined by $y=\pm(1/2)\pi$. 

Note moreover that $\partial_y$ is a Killing vector field for $g_t$ which is the imaginary part of a holomorphic vector field.

From~\eqref{relheighthopf} we obtain 
$$
Q(z)=2t\,\e^z\d z^2
$$
and from~\eqref{eqforgamma} we get $\gamma=\frac{1}{2}$, hence with~\eqref{schwarzwithq} we obtain $P_2(z)=-\frac{3}{8}\d z^2$. Likewise we obtain $P_3(z)=\frac{3}{8}\d z^3$ and $P_4=-\frac{9}{16}\d z^4$ so that
$$
P_4+4(P_2)^2=0. 
$$

\end{ex}

\begin{figure}[h!]
\centering
\includegraphics[trim=0cm 3cm 0cm 1cm,clip=true,scale=0.9]{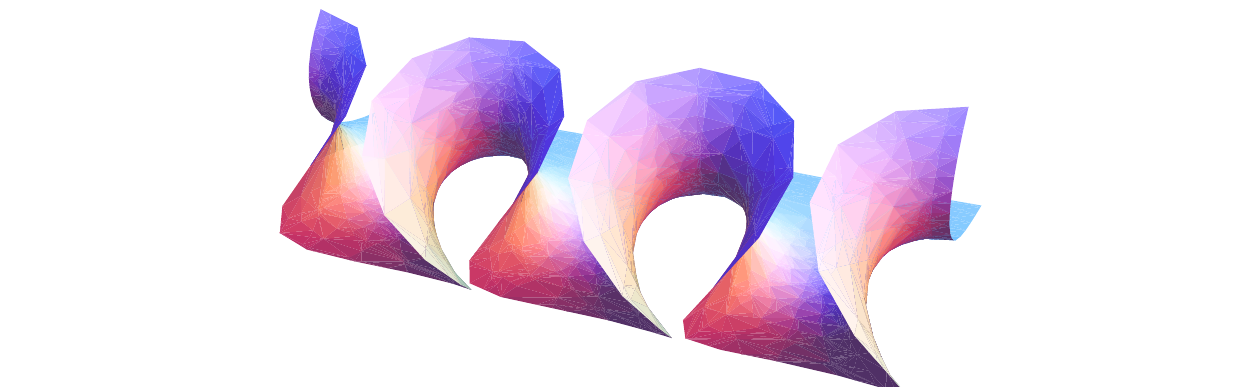}
\caption{A portion of the image of $\bx_1$.}
\end{figure}

\subsection{Degree five surfaces}
A minimal immersion with entropy differentials $\seq$ is of degree $5$ if and only if there exists a complex constant $\lambda$ such that
\[
P_5+\lambda P_2P_3=0.
\]
Therefore, again, we have a (complex) one-dimensional space of algebraic types of degree $5$ immersions and a complete classification is not feasible. However, in entirely similar fashion to the degree four case we can prove:
\begin{ppn}
Suppose $\bx : M\to \mathbb{E}^3$ is a smooth minimal immersion of degree five whose Hopf differential vanishes to order $n\in\mathbb{N}$ at some umbilic point. Then the Hopf differential vanishes to order $n$ at every umbilic point and the immersion $\bx$ satisfies
$$
P_5+24\frac{(n+2)^2}{(3n+4)n}P_3P_2=0. 
$$  
\end{ppn}
The sequence
$$
a_n=24\frac{(n+2)^2}{(3n+4)n}
$$
satisfies $a_1=216/7$ and $a_{\infty}=\lim_{n \to \infty} a_n=8$. Both coefficients $a_1$ and $a_{\infty}$ are realized by well-known surfaces. 

\begin{ex}[Schwarz family] The Schwarz primitive triply periodic surface is the minimal surface with Weierstrass data
\[
M=
\left\{(z,w)\in \hat{\C}^2\,|\,w^2=z^8-14z^4+1\right\}
\]
and $G(z,w)=z$ and $\eta(z,w)=\frac{z}{w}\d z$. From~\eqref{relheighthopf} we compute
$$
Q(z,w)=-\frac{1}{w}\d z^2
$$
and hence using~\eqref{schwarzwithq} we obtain 
$$
P_2(z,w)=-\frac{42z^2(z^4+1)^2}{w^4}\d z^2.
$$
Simple computations give
$$
P_3(z,w)=\frac{84z(z^{16}+34z^{12}-34z^4-1)}{w^6}\d z^3, 
$$
and
$$
P_4(z,w)=-\frac{84}{w^8}(z^{24}+282z^{20}+1887z^{16}-884z^{12}+1887z^8+282z^4+1)\d z^4, 
$$
as well as
$$
P_5(z,w)=\frac{108864z^3(z^{24}+36z^{20}+69z^{16}-69z^8-36z^4-1)}{w^{10}}\d z^5. 
$$
Consequently, we have
$$
P_5+\frac{216}{7}P_2P_3=0, 
$$
thus showing that the Schwarz family of minimal surfaces has degree $5$ and realises the first possible coefficient $\frac{216}{7}$ among degree $5$ surfaces with umbilic points.
\end{ex}

The limit $a_{\infty}=8$ is also realized by a well-known family of minimal surfaces.
\begin{ex}[Scherk family]
The singly-periodic Scherk surface is the minimal surface with Weierstrass data $M=\hat{\C}\setminus\left\{\pm 1,\pm \i\right\}$, $G(z)=z$ and $\eta(z)=\frac{\i z}{(z^4-1)}\d z$. From~\eqref{relheighthopf} we compute
$$
Q(z)=-\frac{\i}{z^4-1}\d z^2
$$
and hence using~\eqref{schwarzwithq} we obtain
$$
P_2(z)=-\frac{6z^2}{(z^4-1)^2}.
$$
Simple computations give
$$
P_3(z)=\frac{12z(z^4+1)}{(z^4-1)^3},\quad P_4(z)=-\frac{12(z^8+10z^4+1)}{(z^4-1)^4} 
$$
and
\[
P_5(z)=\frac{576z^3(z^4+1)}{(z^4-1)^5},
\]
so that
$$
P_5+8P_2P_3=0,
$$
showing that Scherk's family of minimal surfaces has degree $5$ and realizes the limit coefficient $a_{\infty}=8$. 
\end{ex}

\subsection{Degree seven surfaces}
\begin{ex}

Taking $M=\hat{\C}\setminus\{z : z^k=1\}$ and $G(z)=z^{k-1}$ as well as $\eta(z)=\frac{z^{k-1}}{(z^k-1)^2}\d z$ gives the $k$-noids of Jorge \& Meeks~\cite{MR683761}. With the help of a computer algebra system one easily verifies that for $k>2$ the $k$-noids satisfy 
$$
0=P_7+a_kP_5P_2+b_kP_4P_3+c_kP_3(P_2)^2
$$
and hence are surfaces of degree $7$. The coefficients are
$$
\aligned
a_k&=\frac{16(3k-2)(k-2)}{3k^2-16k+8},\\
b_k&=\frac{8(27k^4-144k^3-56k^2+128k-32)}{(3k^2-16k+8)(k-2)(3k-2)},\\
c_k&=\frac{192k^2}{3k^2-16k+8}.
\endaligned
$$
\end{ex}

\begin{rmk}[Minimal surfaces of degree $6$]
The proof of \cref{thmb} given in \cref{sect:approx} allows to conclude that we obtain examples of immersed minimal surfaces of any degree, we were however not able to find well-known examples of embedded minimal surfaces of degree $6$. 
\end{rmk}

%%%%%%%%%%%%%%%%%%%%%%%%%%%%%%%%%%

\section{Approximating minimal surfaces}\label{sect:approx}

It is natural to ask if a minimal surface can be approximated by a sequence of minimal surfaces of increasing degree. In this section we will show that locally and away from umbilic points this is indeed the case. 

Let therefore $M$ be an oriented surface and $\bx : M \to \E^3$ a non-flat minimal immersion. For $p \in M^{\prime}$ pick a simply connected neighbourhood $U_p$ so that $U_p$ does not contain any umbilic points. After possibly shrinking $U_p$ and applying an ambient rotation, we may assume that the stereographically projected Gauss map $G$ of $\bx$ does not attain the value $\infty$ on $U_p$. Let $Q$ denote the Hopf differential and $P_2$ denote the first entropy differential of $\bx$. We take a local $p$-centred holomorphic coordinate $z$ so that $Q=\d z^2$. By analytic continuation we may extend $z$ to all of $U_p$. We write 
$$
\quad P_2=\frac{\rho}{2}\d z^2
$$  
for some holomorphic function $\rho$ on $U_p$, where, by definition, we have
$$
\frac{\rho}{2}=\left\{G,z\right\}.
$$
Recall the following classical fact.
\begin{thm}\label{hillex}
Let $\rho$ be a holomorphic function on a simply connected domain $\Omega$ in the complex plane. Then for every prescription of $w(z_0)$ and $(\partial_z w)(z_0)$ at some point $z_0 \in \Omega$, there exists a unique holomorphic function $w$ on $\Omega$ satisfying Hill's equation with potential $\rho/4$
\begin{equation}\label{hill}
\partial_{zz}w+ \frac{\rho}{4}w=0.
\end{equation}
\end{thm}
For a proof see for instance~\cite[p.~53]{MR867407}. By a straightforward power series coefficient comparison argument, we also obtain: 
\begin{lem}\label{locunifconv}
Let $D\subset \C$ be the open unit disk. For every holomorphic function $\rho$ on $D$ and for all $w_0,w^{\prime}_0 \in \C$, the sequence $w_n : D \to \C$ with $w_n$ being the unique solution of   
$$
\partial^2_{zz}w_n+\frac{\rho_n}{4} w_n=0, \quad w_n(0)=w_0, \quad (\partial_z w_n)(0)=w^{\prime}_0, 
$$
converges locally uniformly to the unique solution of 
$$
\partial^2_{zz}w+\rho w=0, \quad w(0)=w_0, \quad (\partial_z w)(0)=w^{\prime}_0, 
$$
where $\rho_n$ denotes the Taylor series around $0$ of $\rho$ up to order $n$. 
\end{lem}

Hill's equation is linear and by Theorem~\eqref{hillex} its space of solutions is complex two-dimensional. Let therefore $w_1,w_2 : \Omega \to \C$ be a pair of linearly independent solutions. Their Wronskian 
$$
W(w_1,w_2)=w_1\partial_z w_2-w_2\partial_z w_1
$$
is a non-zero constant $W$ by Abel's identity, and $w_1$ and $w_2$ cannot vanish simultaneously. Writing $\hat{G}=\frac{w_1}{w_2}$, we obtain
$$
\partial_z \hat{G}=\frac{w_2\partial_z w_1-w_1\partial_z w_2}{(w_2)^2}=-\frac{W}{(w_2)^2}
$$
so that $\hat{G}$ is a locally injective meromorphic map except possibly on the zero locus of $w_2$. Computing $\partial_z(1/\hat{G})$ we see that $\hat{G}$ is locally injective away from there zero locus of $w_1$, and thus everywhere. A simple computation now gives
$$
\{\hat{G},z\}=-2\frac{\partial^2_{zz}w_2}{w_2}=\frac{\rho}{2},
$$
where we have used the fact that $w_2$ solves~\eqref{hill}. From the cocycle property~\eqref{cocycle} of the Schwarzian derivative we see that $\hat{G}$ differs from the stereographically projected Gauss map by a M\"obius transformation. In particular, on $U_p$, there exists a pair of solutions $w_1,w_2$ of~\eqref{hill} -- unique up to rescaling by a common constant -- so that $G=w_{1}/w_{2}$. Since $G$ does never attain the value $\infty$ on $U_p$ the function $w_2$ is non vanishing on $U_p$. Fix such a pair $(w_1,w_2)$ and let $w_{1,n}$ and $w_{2,n}$ be the unique solutions on $U_p$ to Hill's equation with potential $\rho_n/4$ which satisfy
$$
w_{i,n}(p)=w_i(p), \quad \text{and} \quad (\partial_z w_{i,n})(p)=(\partial_z w_i)(p). 
$$
Here $\rho_n$ denotes the power series of $\rho$ around $p$ of order $(n-3)$ with respect to the coordinate $z$. 
Note that it follows from~\eqref{relheighthopf} and \cref{osswei} that we may recover $\bx : U_p \to \E^3$ by taking the Weierstrass data $(G,-G/(\partial_z G)\d z)$ on $U_p$. However, we may also take $(G_n,-G_n/(\partial_z G_n)\d z)$ as Weierstrass data on $U_p$ where $G_n$ are the locally injective meromorphic functions on $U_p$ defined by
$$
G_n=\frac{w_{1,n}}{w_{2,n}}
$$ 
After possibly shrinking $U_p$ again, we can assume that the functions $G_n$ are holomorphic for sufficiently large $n$, and hence by applying \cref{osswei} again we obtain a sequence of minimal immersions $\left(\bx_n : U_p \to \E^3\right)_{n\geqslant n_0}$.

Recall that $z$ is chosen such that $Q=\d z^2$ and therefore $\gamma\equiv 0$ in \eqref{eqforgamma}. Together with \eqref{defopcoord} this implies that we have
\[
P_{\ell}^{\bx_n}=\frac{1}{2}\partial^{(\ell-2)}_z\rho_n\d z^{\ell},
\]
where $\partial^{(\ell-2)}_z$ denotes the partial derivative of order $(\ell-2)$ with respect to $z$ and $(P_{\ell}^{\bx_n})_{\ell \geqslant 2}$ denote the entropy differentials of the minimal immersion $\bx_n : U_p \to \E^3$. By construction, $\rho_n$ is a polynomial in $z$ of order $(n-3)$ and hence the derivative of $\rho_n$ with respect to $z$ of order $(n-2)$ vanishes identically. This implies that $P_{n}^{\bx_n}$ vanishes identically and hence $\bx_n$ has degree $n$. 

By \cref{locunifconv}, the functions $w_{i,n}$ converge locally uniformly to $w_i$ as $n$ tends to infinity and since $w_2$ is non-vanishing, $G_n$ converges locally uniformly to $G$ as $n$ tends to infinity. Since $\bx_n$ arises by integration from the Weierstrass data $(G_n,-G_n/(\partial_z G_n)\d z)$, it follows that $\Vert\bx_n-\bx\Vert_{g_{Eucl}}$ converges to zero locally uniformly as $n$ tends to infinity. This proves \cref{thmb}.   

%%%%%%%%%%%%%%%%%%%%%%%%%%%%%%%%%

\appendix

\section{Higher order Schwarzian derivatives}\label{appa}

For background about the Schwarzian derivative we refer to reader to~\cite{MR2177471,MR2489717}. We have defined the Schwarzian derivative with respect to a conformal connection. However, less geometric structure is required on a Riemann surface $(M,[g])$ to define the Schwarzian derivative. All one needs is a M\"obius structure~\cite{MR1656822}, which is a generalization of the classical notion of a complex projective structure. A M\"obius structure is a linear second order differential operator which can be written as the sum of the symmetrized trace-free Hessian of a $[g]$-conformal connection $\nabla$ and the real part of a quadratic differential $Q$
$$
\mathrm{H}=\mathrm{Sym}_0(\mathrm{Hess}(\nabla))+\Re(Q). 
$$
The differential operator $\mathrm{H}$ acts on densities of weight $-1/2$ on $M$ and takes values in the symmetric and $[g]$-traceless covariant $2$-tensor fields of weight $-1/2$, i.e.
$$
\mathrm{H} : \Gamma\left(\Lambda^2(T^*M)^{-1/2}\right) \to \Gamma\left(S^2_0(T^*M)\otimes \Lambda^2(T^*M)^{-1/2}\right).
$$ The Schwarzian derivative of a local biholomorphism $\varphi : (M,[g],\mathrm{H}) \to (M^{\prime},[g]^{\prime},\mathrm{H}^{\prime})$ between Riemann surfaces equipped with M\"obius structures is then defined to be
$
\varphi^*\mathrm{H}^{\prime}-\mathrm{H} 
$,
where the pullback is defined in the obvious way. The classical definition of the Schwarzian derivative is recovered by taking both $(M,[g])$ and $(M^{\prime},[g]^{\prime})$ to be $\hat{\C}$ and $\mathrm{H}$ the M\"obius structure associated to the Levi-Civita connection of the spherical metric.  

A M\"obius structure $\mathrm{H}$ on $(M,[g])$ is called flat if in a neighbourhood of every point of $M$ there exists a local holomorphic coordinate $z$ so that 
\begin{equation}\label{flatmoebcoord}
\mathrm{H}=\mathrm{Re}(\partial^2_{zz})
\end{equation}
with respect to the canonical local trivialisations of the vector bundles $\Lambda^2(T^*M)^{-1/2}$ and $S^2(T^*M)\otimes \Lambda^2(T^*M)^{-1/2}$ induced by $z$. We call such a coordinate a local $\mathrm{H}$-adapted coordinate. It is straightforward to check that any two overlapping adapted holomorphic coordinates are related by a M\"obius transformation. 

Recall that a complex projective structure on a Riemann surface $(M,[g])$ is a (maximal) atlas of charts mapping open sets into $\mathbb{CP}^1$ so that transition functions are (restrictions) of M\"obius transformations. Therefore, a flat M\"obius structure induces a complex projective structure and conversely every complex projective structure induces a flat M\"obius structure by defining $\mathrm{H}$ as in~\eqref{flatmoebcoord}. We refer the reader to~\cite{MR1656822} for additional details on M\"obius structures and to~\cite{MR2497780} for complex projective structures.   

A crucial difference between having a conformal connection at hand versus a M\"obius structure only, is when it comes to defining higher order Schwarzian derivatives. The  operators $\mathrm{S}^{\mathrm{H}}_{\ell}$ from \cref{schwarzmoeb} below are well defined on a Riemann surface $(M,[g])$ equipped with a (flat) M\"obius structure. Although they arise as derivatives of the Schwarzian derivative -- and hence may be thought of as higher order Schwarzian derivatives -- for $\ell \geqslant 3$ they all do lack the invariance under post-composition by a M\"obius transformation.  
\begin{ppn}\label{schwarzmoeb}
Let $(M,[g])$ be a Riemann surface equipped with a flat M\"obius structure $\mathrm{H}$. Then for $\ell \geqslant 2$ there exists a differential operator $\mathrm{S}^{\mathrm{H}}_{\ell}$ of order $\ell+1$ mapping locally injective meromorphic functions on $M$ to holomorphic differentials of order $\ell$. In a local $\mathrm{H}$-adapted coordinate $z$ on $M$ the operators are given by
$$
\mathrm{S}^{\mathrm{H}}_{\ell+1}(f)=\left(\partial_z\left(s_{\ell}(f)\right)-\ell s_{\ell}(f)\frac{\partial_{zz}^2f}{\partial_z f}\right)\d z^{\ell+1}
$$
with 
$$\mathrm{S}^{\mathrm{H}}_{2}(f)=\left(\frac{\partial^3_{zzz}f}{\partial_z f}-\left(\frac{\partial^2_{zz}f}{\partial_z f}\right)^2\right)\d z^2
$$
where we write $\mathrm{S}^{\mathrm{H}}_{\ell}(f)=s_{\ell}(f)\d z^{\ell}$.
\end{ppn}
\begin{rmk}
These operators can also be defined in the case where $\mathrm{H}$ is not flat, for the sake of brevity we will however not show this here. In local coordinates, they appeared in~\cite{MR1706981}. 
\end{rmk}
For instance, in a local $\mathrm{H}$-adapted coordinate $z$ we have
$$
\mathrm{S}^{\mathrm{H}}_3(f)=\left[\frac{\partial^4_{zzzz}f}{\partial_z f}-6\,\frac{(\partial^3_{zzz}f)(\partial^2_{zz}f)}{(\partial_z f)^2}+6\left(\frac{\partial^2_{zz}f}{\partial_z f}\right)^3\right]\d z^3.
$$
The reader may easily verify that $\mathrm{S}^{\mathrm{H}}_3(f)$ is not invariant under post-com\-po\-sition by a M\"obius transformation. For comparison, we mention that if $\nabla$ is a $[g]$-conformal connection admitting a local holomorphic coordinate $z$ in which $\frac{1}{2}\gamma^2-\partial_z \gamma=0$, then we have
$$
\mathrm{S}^{\nabla}_3(f)=\left[\frac{\partial^4_{zzzz}f}{\partial_z f}-4\,\frac{(\partial^3_{zzz}f)(\partial^2_{zz}f)}{(\partial_z f)^2}+3\left(\frac{\partial^2_{zz}f}{\partial_z f}\right)^3\right]\d z^3.
$$
\begin{proof}[Proof of \cref{schwarzmoeb}]
We only need to show that the definition of the operators $\mathrm{S}^{\mathrm{H}}_{\ell}$ is invariant under coordinate changes of the form
\begin{equation}\label{coordchange}
w=\frac{az+b}{cz+d}, \quad \text{with} \quad \left(\begin{array}{cc} a & b \\ c & d \end{array}\right)\in \mathrm{SL}(2,\C). 
\end{equation}
We have
$
\d w=\tau\d z$ and $\partial_w=\tau^{-1}\partial_z
$
with
$$
\tau=\frac{(ad-bc)}{(cz+d)^2}.
$$
The proof is by induction. Recall the (well-known) fact that the Schwarzian $\mathrm{S}^{\mathrm{H}}_2$ as defined above is invariant under coordinate changes of the form~\eqref{coordchange}. Let us therefore assume that $\mathrm{S}^{\mathrm{H}}_{\ell}$ is invariant under coordinate changes of the form~\eqref{coordchange}. We write $\mathrm{S}^{\mathrm{H}}_{\ell}(f)=\hat{s}_{\ell}(f)\d w^{\ell}=s_{\ell}(f)\d z^{\ell}$
so that
$$
\hat{s}_{\ell}(f)=\tau^{-\ell}s_{\ell}(f). 
$$
Now using $\partial_z \tau=-\frac{2c}{(cz+d)}\tau$, we obtain
\begin{equation}\label{appida}
\aligned
\partial_w\left(\hat{s}_{\ell}(f)\right)\d w^{\ell+1}&=\left(\tau^{-(\ell+1)}\partial_z (s_{\ell}(f))+\tau^{-1}s_{\ell}(f)\partial_z \tau^{-l}\right)\d w^{\ell+1}\\
&=\partial_{z}(s_{\ell}(f))\d z^{\ell+1}+s_{\ell}(f)\frac{2c}{(cz+d)}\ell \d z^{\ell+1},
\endaligned
\end{equation}
as well as
\begin{equation}\label{appidb}
\aligned
\ell\hat{s}_{\ell}(f)\frac{\partial_w\left(\partial_w f\right)}{\partial_w f}\d w^{\ell+1}&=\ell\tau^{-\ell}s_{\ell}(f)\frac{\partial_z\left(\tau^{-1}\partial_z f\right)}{\partial_z f}\tau^{\ell+1}\d z^{\ell+1}\\
&=\ell s_{\ell}(f)\frac{\partial^2_{zz}(f)}{\partial_z f}\d z^{\ell+1}+\ell \tau s_{\ell}(f)\tau^{-1}\frac{2c}{(cz+d)}\d z^{\ell+1}.
\endaligned
\end{equation}
Combining~\eqref{appida} and ~\eqref{appidb} proves the inductive step. 
\end{proof}
\begin{rmk}
The Schwarzian derivative can be generalized to higher dimensions from the conformal viewpoint~\cite{MR1676999} as well as from the (complex) projective viewpoint~\cite{MR1348154}. 
\end{rmk}

\providecommand{\mr}[1]{\href{http://www.ams.org/mathscinet-getitem?mr=#1}{MR~#1}}
\providecommand{\zbl}[1]{\href{http://www.zentralblatt-math.org/zmath/en/search/?q=an:#1}{zbM~#1}}
\providecommand{\arxiv}[1]{\href{http://www.arxiv.org/abs/#1}{arXiv:#1 }}
\providecommand{\doi}[1]{\href{http://dx.doi.org/#1}{DOI~#1}}
\providecommand{\href}[2]{#2}

\end{document}